\documentclass[12pt,twoside]{amsart}

\usepackage{amsfonts,amsmath,amsthm,amssymb,latexsym,amsthm,newlfont,enumerate,color,amscd}

\newif\ifslide

\theoremstyle{plain}

\ifdefined\spacer
\newtheorem{theorem}{Theorem}
 \else
\newtheorem{theorem}{Theorem}[section]
\fi

\newtheorem{lemma}[theorem]{Lemma}

\newtheorem*{theorem*}{Theorem}

\newtheorem{proposition}[theorem]{Proposition}
\newtheorem{definition-lemma}[theorem]{Definition-Lemma}

\newtheorem{red-question}[theorem]{\textcolor{red}{Question}}

\theoremstyle{definition}
\newtheorem{definition}[theorem]{Definition}

\def\ideal#1.{I_{#1}}
\def\ring#1.{\mathcal {O}_{#1}}
\def\Spec{\operatorname {Spec}}

\def\fring#1.{\hat{\mathcal {O}}_{#1}}
\def\proj#1.{\mathbb {P}(#1)}
\def\pr #1.{\mathbb {P}^{#1}}
\def\dpr #1.{\hat{\mathbb {P}}^{#1}}
\def\af #1.{\mathbb A^{#1}}
\def\Hz #1.{\mathbb F_{#1}}
\def\Hbz #1.{\overline{\mathbb F}_{#1}}
\def\fb#1.{\underset #1 {\times}}
\def\rest#1.{\underset {\ \ring #1.} \to \otimes}
\def\au#1.{\operatorname {Aut}\,(#1)}
\def\deg#1.{\operatorname {deg } (#1)}

\def\pic#1.{\operatorname {Pic}\,(#1)}
\def\pico#1.{\operatorname{Pic}^0(#1)}
\def\picg#1.{\operatorname {Pic}^G(#1)}
\def\ner#1.{NS (#1)}
\def\rdown#1.{\llcorner#1\lrcorner}
\def\rfdown#1.{\lfloor{#1}\rfloor}
\def\rup#1.{\ulcorner{#1}\urcorner}
\def\rcup#1.{\lceil{#1}\rceil}

\def\n1#1.{\operatorname {N_1}(#1)}  
\def\cn1#1.{\overline{\operatorname {N^1}(#1)}} 
\def\cone#1.{\operatorname {NE}(#1)}     
\def\ccone#1.{\overline{\operatorname {NE}}(#1)}
\def\none#1.{\operatorname {NF}(#1)}
\def\cnone#1.{\overline{\operatorname {NF}}(#1)}
\def\mone#1.{\operatorname {NM}(#1)} 
\def\cmone#1.{\overline{\operatorname {NM}}(#1)}

\def\coef#1.{\frac{(#1-1)}{#1}}
\def\vit#1.{D_{\langle #1 \rangle}}
\def\mm#1.{\overline {M}_{0,#1}}
\def\H1#1.{H^1(#1,{\ring #1.})}
\def\ac#1.{\overline {\mathbb F}_{#1}}

\def\adj#1.{\frac {#1-1}{#1}}
\def\spn#1.{\overline{#1}}
\def\pek#1.#2.{\Cal P^{#1}(#2)}
\def\plk#1.#2.{\Cal P^{\leq #1}(#2)}
\def\ev#1.{\operatorname{ev_{#1}}}
\def\ilist#1.{{#1}_1,{#1}_2,\dots}
\def\bminv#1.{(\nu_1,s_1;\nu_2,s_2;\dots ;\nu_{#1},s_{#1};\nu_{r+1})}
\def\zinv#1.{(\nu_1,s_1;\nu_2,s_2;\dots ;\nu_{#1},s_{#1};0)}
\def\iinv#1.{(\nu_1,s_1;\nu_2,s_2;\dots ;\nu_{#1},s_{#1};\infty)}

\def\scr #1.{\mathcal #1}

\def\llist#1.#2.{{#1}_1,{#1}_2,\dots,{#1}_{#2}}
\def\ulist#1.#2.{{#1}^1,{#1}^2,\dots,{#1}^{#2}}
\def\lomitlist#1.#2.{{#1}_1,{#1}_2,\dots,\hat {{#1}_i}, \dots, {#1}_{#2}}
\def\lomitlistz#1.#2.{{#1}_0,{#1}_1,\dots,\hat {{#1}_i}, \dots, {#1}_{#2}}
\def\loc#1.#2.{\Cal O_{#1,#2}}
\def\fderiv#1.#2.{\frac {\partial #1}{\partial #2}}
\def\deriv#1.#2.{\frac {d #1}{d #2}}

\def\map#1.#2.{#1 \longrightarrow #2}
\def\rmap#1.#2.{#1 \dasharrow #2}
\def\emb#1.#2.{#1 \hookrightarrow #2}
\def\non#1.#2.{\text {Spec }#1[\epsilon]/(\epsilon)^{#2}}
\def\Hi#1.#2.{\text {Hilb}^{#1}(#2)}
\def\sym#1.#2.{\operatorname {Sym}^{#1}(#2)}
\def\Hb#1.#2.{\text {Hilb}_{#1}(#2)}
\def\Hm#1.#2.{\Hom_{#1}(#2)}
\def\prd#1.#2.{{#1}_1\cdot {#1}_2\cdots {#1}_{#2}}
\def\Bl #1.#2.{\operatorname {Bl}_{#1}#2}
\def\pl #1.#2.{#1^{\otimes #2}}
\def\mgn#1.#2.{\overline {M}_{#1,#2}}
\def\ialist#1.#2.{{#1}_1 #2 {#1}_2, #2\dots}
\def\pair#1.#2.{\langle #1, #2\rangle}
\def\vandermonde#1.#2.{\left|
\begin{matrix}
1 & 1 & 1 & \dots & 1\\
{#1}_1 & {#1}_2 & {#1}_3 & \dots & {#1}_{#2}\\
{#1}_1^2 & {#1}_2^2 & {#1}_3^2 & \dots & {#1}_{#2}^2\\
\vdots & \vdots & \vdots & \ddots & \vdots\\
{#1}_1^{#2-1} & {#1}_2^{#2-1} & {#1}_2^{#2-1} & \dots & {#1}_{#2}^{#2-1}\\
\end{matrix}
\right|
}
\def\vandermondet#1.#2.{\left|
\begin{matrix}
1 & {#1}_1   & {#1}_1^2 & \dots & {#1}_1^{#2-1}\\
1 & {#1}_2   & {#1}_2^2 & \dots & {#1}_2^{#2-1}\\
1 & {#1}_3   & {#1}_3^2 & \dots & {#1}_3^{#2-1}\\
\vdots & \vdots & \vdots & \ddots & \vdots\\
1 & {#1}_{#2}& {#1}_{#2}^2 & \dots & {#1}_{#2}^{#2-1}\\
\end{matrix}
\right|
}
\def\gr#1.#2.{\mathbb{G}(#1,#2)}

\def\alist#1.#2.#3.{{#1}_1 #2 {#1}_2 #2\dots #2 {#1}_{#3}}
\def\zlist#1.#2.#3.{#1_0 #2 #1_1 #2\dots #2 #1_{#3}}
\def\lomitlist30#1.#2.#3.{{#1}_0,{#1}_1 #2 \dots #2\hat {{#1}_i} #2\dots #2 {#1}_{#3}}
\def\lmap#1.#2.#3.{#1 \overset{#2}{\longrightarrow} #3}
\def\mes#1.#2.#3.{#1 \longrightarrow #2 \longrightarrow #3}
\def\ses#1.#2.#3.{0\longrightarrow #1 \longrightarrow #2 \longrightarrow #3 \longrightarrow 0}
\def\les#1.#2.#3.{0\longrightarrow #1 \longrightarrow #2 \longrightarrow #3}
\def\res#1.#2.#3.{#1 \longrightarrow #2 \longrightarrow #3\longrightarrow 0}
\def\Hi#1.#2.#3.{\text {Hilb}^{#1}_{#2}(#3)}
\def\ten#1.#2.#3.{#1\underset {#2}{\otimes} #3}
\def\lomitlist30#1.#2.#3.{{#1}_0 #2 {#1}_1 #2 \dots #2 \hat {{#1}_i} #2 \dots #2 {#1}_{#3}}
\def\mderiv#1.#2.#3.{\frac {d^{#3} #1}{d #2^{#3}}}

\def\Diff{\operatorname{Diff}}

\def\Hom{\operatorname{Hom}}

\def\Proj{\operatorname{Proj}}

\def\Supp{\operatorname{Supp}}

\def\deg{\operatorname{deg}}

\def\rest{\operatorname{res}}

\def\e{\Cal E}

\def\e1{E_1}
\def\e2{E_2}

\def\Q{\mathbb Q}
\def\Z{\mathbb Z}

\def\mapdown#1{\big\downarrow\rlap{$\vcenter{\hbox{$\scriptstyle#1$}}$}}

\def\mapse#1{
{\vcenter{\hbox{$\mathop{\smash{\raise1pt\hbox{$\diagdown$}\!\lower7pt
\hbox{$\searrow$}}\vphantom{p}}\limits_{#1}\vphantom{\mapdown{}}$}}}}

\def\VR#1.{height#1pt&\omit&&\omit&&\omit&&\omit&&\omit&\cr}

\def\VRT#1.{height#1pt&\omit&&\omit&\cr}

\usepackage{hyperref}

\usepackage{mathtools}
\usepackage{tikz}
\usetikzlibrary{matrix,arrows,decorations.pathmorphing, shapes, calc}

\usepackage{color}
\usepackage[normalem]{ulem} 
\newcommand\redout{\bgroup\markoverwith
{\textcolor{red}{\rule[.5ex]{2pt}{0.4pt}}}\ULon}
\usepackage{xcolor}
\usepackage{tikz-cd}
\usepackage{tikz}
\usetikzlibrary{matrix,arrows,decorations.pathmorphing}
\usepackage{comment}
\usepackage{enumitem}
\usepackage{microtype}

\title[Global F-regularity for log del Pezzo surfaces]
{Klt del Pezzo surfaces which are not globally F-split} 
\author{Paolo Cascini, Hiromu Tanaka, Jakub Witaszek} 
\address{Department of Mathematics, Imperial College, London, 180 Queen's Gate, 
London SW7 2AZ, UK} 
\email{p.cascini@imperial.ac.uk}
\email{h.tanaka@imperial.ac.uk}
\email{j.witaszek14@imperial.ac.uk}

\newtheorem{step}{Step}

\usepackage{mathtools}

\def\vdots{\vbox{\baselineskip=6pt \lineskiplimit=0pt 
\kern6pt \hbox{.}\hbox{.}\hbox{.}}} 

\thanks{All three authors  were funded by  EPSRC}

\begin{document}

\maketitle

\begin{abstract} We construct a klt del Pezzo surface which is not globally $F$-split, over any algebraically closed field of positive characteristic.
\end{abstract}

\tableofcontents

\setcounter{section}{0}

\section{Introduction}
It has been clear for a long time that the right category of singularities to work with, in trying to run the minimal model programme in characteristic zero, is the category of Kawamata log terminal pairs  -- klt, for short.

The main difficulty in trying to extend the classic results in minimal model programme, such as the base point free theorem and the cone theorem,  to any projective variety defined over an algebraically closed field of positive characteristic is the fact that the Kawamata--Viehweg vanishing theorem does not hold in this generality. On the other hand, over the last few years, many methods have been developed to use the Frobenius morphism to lift sections, as a replacement of the vanishing theorems. As a consequence,  $F$-singularities are more suitable singularities to work with in terms of generalising some of the results in birational geometry to varieties defined over a field of positive characteristic. These singularities were introduced in \cite{hh90} to develop new techniques in commutative algebra.

The recent progress in the field has been sparked by the introduction of strongly $F$-regular singularities
 and globally $F$-regular varieties (\cite{smith00} and \cite{schwedesmith10}). Those notions are pereceived as correct counterparts of klt singularities and klt Fano varieties, respectively.
For example, Watanabe showed that over an algebraically closed field of characteristic $p>5$, 
any one dimensional log Fano pair $(X, \Delta)$ with standard coefficients is globally $F$-regular 
\cite[Theorem 4.2]{watanabe91}. 
This result played an important role in the  work of Hacon and Xu, who showed 
the existence of minimal models for any projective $\mathbb Q$-factorial terminal threefold with pseudo-effective canonical divisor, defined over an algebraically closed field of characteristic $p>5$ \cite{hx13}.

Later on, the first author, together with Gongyo and Schwede, showed that a  two-dimensional klt singularity $(X,\Delta)$ is strongly $F$-regular for big enough characteristic depending only on the coefficients of $\Delta$ \cite{cgs14}.

The goal of this paper is to show that the natural generalization of the results in \cite{hara98a, watanabe91, cgs14} do not hold true even under the assumption that the boundary divisor is trivial  (cf.\ \cite[Question 5.23]{PST14}). Our main result is:
\begin{theorem} \label{Intro-main} Over any  algebraically closed field $k$ of characteristic $p>0$, there exists a projective Kawamata log terminal surface $X$  such that $-K_X$ is ample and $X$ is not globally $F$-split.
\end{theorem}
 Note that, if $X$ is a smooth del Pezzo surface  over an algebraically closed field of characteristic  $p>5$, then $X$ is globally $F$-regular \cite[Example 5.5]{hara98a}.
In \cite{CTW15b}, we show that any Kawamata log terminal del Pezzo surface over an algebraically closed field of large characteristic which is not globally $F$-regular, admits a log resolution which is liftable to characteristic zero. 
\medskip 

By considering the cone over a klt del Pezzo surface which is not globally $F$-split, as in Theorem \ref{Intro-main},
we obtain a three-dimensional klt singularity in arbitrary characteristic which is not $F$-pure. 
On the other hand, we give a direct and simpler construction of such singularities. In particular, the varieties that we construct, admit only  canonical hypersurface singularities. 

\begin{theorem}\label{Intro-main2}
Over any  algebraically closed field $k$ of characteristic $p>0$,
there exists a three-dimensional canonical variety $X$, which is not $F$-pure. 
\end{theorem}


Theorem~\ref{Intro-main} and Theorem~\ref{Intro-main2} 
are proved in Section~\ref{s-2dim} and Section~\ref{s-3dim}, respectively.

\medskip
\textbf{Acknowledgement: } 
We would like to thank Y. Gongyo, Y. Kawamata, Z. Patakfalvi, A. Sannai, K. Schwede and S. Takagi for many useful discussions and comments. 

\section{Preliminary Results}
\subsection{Notation}

We work over an algebraically closed field $k$ of characteristic $p>0$, 
unless otherwise mentioned. 
We say that a scheme $X$ is {\em of characteristic} $p>0$ if 
the natural morphism $X \to \mathrm{Spec}\,\Z$ factors through $\mathrm{Spec}\,\Z/p\Z$. 
For any prime $p$, we denote by $\Z_{(p)}$ the set of rational numbers whose denominator is not divisible by $p$. 

We say that $X$ is a {\em variety} over an algebraically closed field $k$, if $X$ is an integral scheme which is separated and of finite type over $k$. 
A {\em curve} is a variety of dimension one, and a {\em surface} is a variety of dimension two. 

We refer to \cite{KM98} for the classical definitions of singularities (for example {\it klt, plt, dlt and log canonical}) appearing in the minimal model programme. 
For a $\mathbb Q$-Cartier $\mathbb Q$-divisor $D$ on a normal variety, 
the {\em Cartier index} of $D$ is the smallest positive integer $m$ such that $mD$ is Cartier. 

Given a scheme $X$ of characteristic $p>0$,
we denote by $F\colon X\to X$ the absolute Frobenius morphism, and we say that $X$ is $F$-{\em finite} if  $F_*\mathcal O_X$ is a finitely generated $\mathcal O_X$-module.
We now recall the definitions of some classes of $F$-singularities (e.g.\ see \cite[Definition 3.1]{schwedesmith10}): 

\begin{definition}
Let $X$ be an $F$-finite normal scheme of characteristic $p>0$ 
(for example a normal variety over an algebraically closed field) 
and let $\Delta$ be an effective $\mathbb R$-divisor on $X$. 
\begin{enumerate}
\item{The pair $(X, \Delta)$ is {\em globally $F$-split} if for every $e \in \mathbb Z_{>0}$, 
the natural map  
\[
\mathcal O_X \xrightarrow{F^e} F_*^e\mathcal O_X \xhookrightarrow{\hphantom{F^e}} F_*^e\mathcal O_X(\llcorner (p^e-1)\Delta\lrcorner)
\]
splits as an $\mathcal O_X$-module homomorphism.}
\item{The pair $(X, \Delta)$ is {\em globally sharply $F$-split} if there exists $e \in \mathbb Z_{>0}$, such that 
the natural map  
\[
\mathcal O_X \xrightarrow{F^e} F_*^e\mathcal O_X \xhookrightarrow{\hphantom{F^e}} F_*^e\mathcal O_X(\ulcorner (p^e-1)\Delta\urcorner)
\]
splits as an $\mathcal O_X$-module homomorphism.}
\item{The pair $(X, \Delta)$ is {\em globally $F$-regular} if 
for every effective divisor $E$ on $X$, there exists $e \in \mathbb Z_{>0}$
such that the natural map 
\[
\mathcal O_X \xrightarrow{F^e} F_*^e\mathcal O_X \xhookrightarrow{\hphantom{F^e}} F_*^e\mathcal O_X(\llcorner (p^e-1)\Delta\lrcorner+E)
\]
splits as an $\mathcal O_X$-module homomorphism.}
\item{The pair $(X, \Delta)$ is  {\em $F$-pure} (resp.\ {\em sharply $F$-pure}, {\em strongly $F$-regular}) 
if $(\Spec\,\mathcal O_{X, x}, \Delta|_{\Spec\,\mathcal O_{X, x}})$ is globally $F$-split 
(resp.\ globally sharply F-split, globally $F$-regular) for every point $x\in X$.}
\end{enumerate}
\end{definition}
If the coefficients of $\Delta$ belong to $\mathbb Z_{(p)}$, then being sharply $F$-split coincides with being $F$-split.

\subsection{Preliminaries}
We begin by recalling the following basic result: 

\begin{lemma} \label{l-push-GFS} 
Let $f\colon X \to Y$ be a proper birational morphism 
of normal varieties. 
Let $\Delta$ be an effective $\Q$-divisor on $X$. 
If $(X, \Delta)$ is globally $F$-split, then so is $(Y, f_*\Delta)$. 
\end{lemma}

\begin{proof}
Let  $e\in \Z_{>0}$. 
Since $(X, \Delta)$ is globally $F$-split, the Frobenius homomorphism induced by $(X,\Delta)$ splits: 
\[
\mathrm{id}\colon \mathcal{O}_X \xrightarrow{F^e} F_*^e\mathcal{O}_X(\llcorner (p^e-1)\Delta\lrcorner) 
\xrightarrow{\exists \,\varphi} \mathcal O_X.
\]
Consider  an open subscheme $j \colon U \hookrightarrow Y$ of codimension at least two, 
such that the restriction of $f$ to $U$ induces an isomorphism $f^{-1}(U) \to U$. By restricting to $U$, we have that the Frobenius homomorphism 
splits  on $U$:
\[
\mathrm{id} \colon \mathcal{O}_U \xrightarrow{F^e} F_*^e\mathcal{O}_U(\llcorner (p^e-1)f_*\Delta \lrcorner) 
\xrightarrow{\varphi |_U} \mathcal O_U.
\]
Since $U$ is of codimension at least two, the Lemma follows by taking the pushforward by $j_*$. 
\end{proof}

%
%
By a result of  Hacon and Xu, also the opposite implication holds true:
\begin{lemma}[{\cite[Prop.~2.11]{hx13}}]\label{l-pull-GFS} Let $(Y,\Delta)$ be a globally $F$-regular pair and $f \colon X \to Y$ a proper birational morphism of normal varieties such that $f^*(K_Y + \Delta) = K_X + \Delta'$, for some $\mathbb Q$-divisor $\Delta'\ge 0$.

 Then, $(X,\Delta')$ is globally $F$-regular.
\end{lemma}

\medskip

The following result is a consequence of \cite{watanabe91} and an elementary calculation. 

\begin{lemma}\label{l_P1}
Let $k$ be an algebraically closed field of characteristic ${p > 0}$. 
Then, there exist   four distinct closed points $Q_1, \ldots, Q_4 \in \mathbb P^1_k$ 
such that the pair 
\[
\Big(\mathbb P_k^1, \frac{1}{2}\sum_{i=1}^4 Q_i\Big)
\]
is not globally $F$-split.  
\end{lemma}

\begin{proof}
If $p=2$, then  \cite[Thm.~4.2(h)]{watanabe91} implies that the pair is not globally $F$-split, for any choice of the points $Q_1,\dots,Q_4$. Thus, we may write $p=2n+1$ for some $n \in \mathbb Z_{>0}$. 
Consider the four distinct points of $\mathbb P^1_k$: 
\begin{align*}
 Q_1 &= [1,0], \ Q_2 = [0,1], \ Q_3 = [-1,1], \text{ and } Q_4 = [-\mu, 1]
\end{align*}
where $\mu \in k \, \backslash \, \{0,1\}$. By \cite[Thm.~4.2(h)]{watanabe91}, 
the  pair 
\[
\Big(\mathbb P_k^1, \frac{1}{2}\sum_{i=1}^4 Q_i\Big )
\]
is globally $F$-split if and only if the coefficient of $x^n$ in the expansion of $(x+1)^n(x+\mu)^n$ is not zero. 

The coefficient of $x^n$ can be expressed in terms of the following monic polynomial in the variable $\mu$: 
\[
\sum_{i=0}^n\binom{n} {i}  \binom{n} {n-i}\mu^i= \mu^n+n^2\mu^{n-1}+\ldots+1 \in k[\mu].
\]
Since $k$ is algebraically closed, it is enough to show that such a  polynomial admits a solution which is different from  zero and one. 
By contradiction, assume that 
\[
\mu^n+n^2\mu^{n-1}+\ldots+1=\mu^a(\mu-1)^b,
\]
for some non-negative integers $a$ and $b$
satisfying $a+b=n$. 
Clearly $a=0$, hence we obtain 
\[
\mu^n+n^2\mu^{n-1}+\ldots+1=(\mu-1)^n\,\,\,\,\mathrm{in}\,\,\,\,k[\mu].
\]
It follows that $n^2\equiv -n\mod p$ and, in 
particular, $p$ divides either $n$ or $n+1$. 
Since $n+1<2n+1=p$, we get a contradiction. Thus, the claim follows. 
\end{proof}

Let $X$ be a normal surface over an algebraically closed field $k$. Let $C$ be a smooth prime divisor on $X$ such that $(X,C)$ is log canonical. Then there exists an effective $\mathbb Q$-divisor $\Diff_C$ on $C$, called the {\em different}, such that 
\[
(K_X+C)|_C=K_C+\Diff_C
\]
\cite[Prop.-Def. 16.5]{Kollaretal}.

\begin{lemma} \label{l_adjunction} Let $(Z,C)$ be a  two-dimensional projective plt pair  over an algebraically closed field of characteristic $p>0$, such that $C$ is a smooth prime divisor. 
If $(C, \mathrm{Diff}_{C})$ is not globally $F$-split and the Cartier index of $K_Z + C$ is not divisible by $p$, 
then $(Z,C)$ is not globally $F$-split.
\end{lemma}
\begin{proof}
Let $e \in \mathbb{Z}_{>0}$ be such that $(p^e-1)(K_Z + C)$ is Cartier. It is enough to show that the trace map 
\[
\textrm{Tr}^e_Z \colon H^0\left(Z, F^e_*\mathcal O_Z \left(-(p^e-1)(K_Z+C)\right)\right) \to H^0(Z, \mathcal O_Z)
\]
is zero. 

Let $\mathcal{L} := \mathcal O_Z(-(p^e-1)(K_Z+C))$. Then 
\[
\mathcal{L}|_C = \mathcal O_C(-(p^e-1)(K_C+\Diff_C))
\]
Consider the following commutative diagram:
\begin{center}
\begin{tikzcd}
H^0(Z, F_*^e\mathcal{L}) \arrow{r} \arrow{d}{\textrm{Tr}_Z^e} &  
H^0(C, F_*^e(\mathcal{L}|_C)) \arrow{d}{\textrm{Tr}_{C}^e} \\
H^0(Z, \mathcal O_Z) \arrow{r}{\simeq} &
H^0(C, \mathcal O_{C}).
\end{tikzcd}
\end{center}
Since $(C, \Diff_C)$ is not globally $F$-split, the right vertical arrow $\mathrm{Tr}^e_C$ is zero. Hence, $\mathrm{Tr}^e_Z$ is zero, and the claim follows.
\end{proof}

\section{F-purity on the F-pure threshold}

%
%

The goal of this section is to prove Theorem \ref{t-fpure-threshold}, 
which implies that being globally $F$-split is preserved after taking a limit. 

The result was already known in  some special cases
(e.g. see \cite[Proposition~2.6]{hara06}, \cite[Theorem 4.1]{hernandez12}
 and 
\cite[Remark 5.5]{schwede08}). 
On the other hand, if we consider a pair $(X, \mathfrak{a})$ consisting of a normal variety $X$ and 
an ideal $\mathfrak{a}$, instead of a divisor, 
then the result does not hold \cite[Proposition 2.2(6) and Remark 2.3(2)]{tw04}. 
Moreover, note that being globally sharply $F$-split is not preserved after taking a limit \cite[Remark 6.2]{schwedesmith10}. 

We begin  by recalling the definitions of the  $F$-pure and the $F$-split thresholds. 
The first notion was introduced by  Takagi and Watanabe in \cite{tw04}. We refer to  \cite{BMS08, BSTZ10,HNWZ15} for more results in this direction.

\begin{definition}
Let $X$ be an $F$-finite normal scheme of characteristic $p>0$ and 
let $\Delta$ and $D$ be effective $\mathbb R$-divisors. 
\begin{enumerate}
\item We define the $F$-{\em split threshold} of $(X,\Delta)$ with respect to $D$ as\\
$\displaystyle\mathrm{fst}(X, \Delta; D):=\sup\{ t\in \mathbb R_{\geq 0}\,|\, (X, \Delta+tD) \text{ is globally $F$-split}\}$.
\item We define the $F$-{\em pure threshold} of $(X,\Delta)$ with respect to $D$ as\\ 
$\displaystyle\mathrm{fpt}(X, \Delta; D) :=\sup\{ t\in \mathbb R_{\geq 0}\,|\, (X, \Delta+tD) \text{ is  $F$-pure}\}$.
\end{enumerate}

\end{definition}

The following Theorem is used in Section \ref{s-2dim} to prove Theorem \ref{Intro-main}:

\begin{theorem}\label{t-fpure-threshold}
Let $X$ be an $F$-finite noetherian normal scheme of characteristic $p>0$ and let $\Delta$ be an  $\mathbb R$-divisor 
such that $(X,\Delta)$ is globally $F$-split. Let $D$ be an effective $\mathbb R$-divisor on $X$ and let $c=\mathrm{fst}(X, \Delta; D)$.

Then $(X, \Delta+cD)$ is globally $F$-split.
\end{theorem}

\begin{proof}
If $c=0$, then there is nothing to show. 
Thus, we may assume $c>0$. We divide the proof into three steps:
 
\setcounter{step}{0}

\begin{step}\label{step-nonZp}

We first show the Theorem, assuming that the following property holds:
\begin{enumerate}
\item{$D$ and $\Delta$ do not have any common component, 
 $D$  has coefficients in $\mathbb R \setminus \mathbb Z_{(p)}$, and 
$\mathrm{fst}(X, \Delta; D) =1$.}
\end{enumerate} 
\end{step}

Let $e\in \Z_{>0}$. 
We need to show that 
\[
\mathcal O_X \longrightarrow F_*^e\mathcal O_X(\llcorner (p^e-1)(\Delta+D) \lrcorner)
\]
splits. 
Since $D$ and $\Delta$ do not have any common component, we have 
\[
\llcorner (p^e-1)(\Delta+D) \lrcorner=\llcorner (p^e-1)\Delta\lrcorner+\llcorner (p^e-1)D\lrcorner.
\]
Since the coefficients of $D$ are contained in $\mathbb R \setminus \mathbb Z_{(p)}$, 
there exists $t\in (0,1)$ such that 
\[
\llcorner (p^e-1)D \lrcorner = \llcorner (p^e-1)t D \lrcorner.
\]
Since $\mathrm{fst}(X,\Delta;D)=1$, the claim follows. 

\begin{step}\label{step-Zp}
We now show the Theorem assuming that the following property holds: 
\begin{enumerate}
\item[(2)] {$D$ is a prime divisor, $D \not\subset \Supp \Delta$, and $\Delta$ has coefficients in  $\mathbb Z_{(p)}$. 
Moreover, $\mathrm{fst}(X, \Delta; D) \in \mathbb Z_{(p)}$}.
\end{enumerate}
\end{step}

Since the coefficients of $\Delta+cD$ are contained in $ \mathbb Z_{(p)}$, it is enough to show that there exists $e\in \mathbb Z_{>0}$ such that 
\[
\mathcal O_X \longrightarrow F_*^e\mathcal O_X(\llcorner (p^e-1)(\Delta+cD) \lrcorner)
\]
splits. 
Consider the subset $N_1 \subset \mathbb Z_{>0}$ such that $e \in N_1$ 
if and only if $(p^e-1)\Delta$ has coefficients in  $\mathbb Z$ and $(p^e-1)c$ is an integer. 
For any $e \in N_1$, we define 
\[
\nu(e):=\max\left\{s \in \mathbb Z_{\geq 0}\,\middle|\,  \mathcal O_X \rightarrow F_*^e\mathcal O_X((p^e-1)\Delta+sD)\text{ splits}\right\}.
\]
By the definition of $\nu(e)$ and $c$, 
we see that 
\[
\nu(e) \leq (p^e-1)c \leq \nu(e)+1.
\]
Thus,  
\[
\nu(e) \in \{(p^e-1)c-1, (p^e-1)c\}.
\]
If $\nu(e)=(p^e-1)c$ for some $e \in N_1$, 
then $(X, \Delta+cD)$ is globally $F$-split and we are done. 
We now assume by contradiction that 
\[
\nu(e)=(p^e-1)c-1.
\]
for every $e \in N_1$.

Pick $e \in N_1$. Then, $2e\in N_1$.  Consider 
\[
\mathcal O_X \rightarrow F_*^e\mathcal O_X((p^e-1)\Delta).
\]
After tensoring by $\mathcal O_X((p^e-1)\Delta+(\nu(e)+1)D)$, taking the double dual, and applying $F_*^e$, we obtain: 
\begin{eqnarray*}
F_*^e\mathcal O_X((p^e-1)\Delta+(\nu(e)+1)D) 
\to F_*^{2e}\mathcal O_X((p^{2e}-1)\Delta+p^e(\nu(e)+1)D).
\end{eqnarray*}
Since 
\[
\mathcal O_X \rightarrow F_*^e\mathcal O_X((p^e-1)\Delta+(\nu(e)+1)D)
\]
does not split, neither does 
\[
\mathcal O_X \rightarrow F_*^{2e}\mathcal O_X((p^{2e}-1)\Delta+p^e(\nu(e)+1)D).
\]
Thus,
\[
\nu(2e)<p^e(\nu(e)+1).
\]
Since $\nu(d)=(p^d-1)c-1$ for $d\in \{e, 2e\}$, and since  $(p^e-1)c \in \mathbb Z_{>0}$, we have 
\[
(p^{2e}-1)c-1 <p^e(p^e-1)c \leq p^e(p^e-1)c+(p^e-1)c-1=(p^{2e}-1)c-1.
\]
This is a contradiction. Thus, the claim follows.

\begin{step}\label{step-reduction}
We now show the Theorem in the general case. 
\end{step}

We may write 
\[
\Delta+cD=\sum_{i=1}^r a_iA_i+\sum_{j=1}^s b_jB_j
\]
where $a_i \in \mathbb Z_{(p)}$, $b_j \in \mathbb R\setminus \mathbb Z_{(p)}$ and $A_1,\dots,A_r,B_1,\dots,B_s$ are irreducible components. 

We first show that $(X, \sum_{i=1}^r a_iA_i)$ is globally $F$-split. 
If not, then by  Step 2, we can find $a'_1 \in \mathbb Z_{(p)}$ such that $0 \leq a'_1<a_1$ and 
 $(X, a'_1A_1+\sum_{i=2}^r a_iA_i)$ is not globally $F$-split. 
Proceeding this way, for every $1\leq i \leq r$, 
we can find $a'_i \in \mathbb Z_{(p)}$ such that $0 \leq a'_i<a_i$ and 
 $(X, \sum_{i=1}^r a'_iA_i)$ is not globally $F$-split. 
 In particular, there exists $\delta\in (0,c)$ such that
 \[
\sum_{i=1}^r a'_iA_i \le \Delta+(c-\delta)D. 
 \]
This contradicts the fact that $(X, \Delta+(c-\delta)D)$ is globally $F$-split for any $\delta\in (0,c)$. 
Thus, $(X, \sum_{i=1}^r a_iA_i)$ is globally $F$-split.

Let
\[
c':=\mathrm{fst}(X, \sum_{i=1}^r a_iA_i; \sum_{j=1}^s b_jB_j).
\]
If $c'>1$, then $(X,\Delta+cD)$ is globally $F$-split, by the definition of the $F$-split threshold. If $c'=1$, then the same claim follows from  (2) of Step 1. 

We assume now by  contradiction that $0 \leq c'<1$. For each $j=1,\dots,s$, 
there exists $b'_j \in \mathbb Z_{(p)}$ such that $c'b_j<b_j'<b_j$. 
In particular, the pair 
\[
(X, \sum_{i=1}^ra_i A_i+\sum_{j=1}^sb_j'B_j)
\]
is not globally $F$-split. 
By Step 2, it follows that 
$$(X, a_1'A_1+\sum_{i=2}^ra_i A_i+\sum_{j=1}^sb_j'B_j)$$
is not globally $F$-split for some $a'_1 \in \mathbb Z_{(p)}$ with $0\leq a'_1<a_1$. 
Proceeding this way, we see that 
 $$(X, \sum_{i=1}^ra_i'A_i+\sum_{j=1}^sb_j'B_j)$$
is not globally $F$-split for some $a'_i \in \mathbb Z_{(p)}$ with $0\leq a'_i<a_i$ for $i=1,\dots,r$. 
We may find $\delta>0$ such that:  
\[
\begin{aligned}
\sum_{i=1}^ra_i'A_1+\sum_{j=1}^sb_j'B_j \leq& (1-\delta)\left(\sum_{i=1}^ra_iA_i+c\sum_{j=1}^sb_jB_j\right)\\
=&(1-\delta)(\Delta+cD) \leq \Delta+(1-\delta)cD.
\end{aligned}
\]
Thus, $c=\mathrm{fst}(X,\Delta;D)\le (1-\delta)c$, a contradiction. 
\end{proof}


As an immediate consequence, we obtain: 

\begin{proposition} Let $X$ be an $F$-finite noetherian normal scheme of characteristic $p>0$ and 
let $\Delta$ be an  $\mathbb R$-divisor such that $(X, \Delta)$ is $F$-pure. 
Let $D$ be an effective $\mathbb R$-divisor on $X$ and let $c=\mathrm{fpt}(X,\Delta;D)$. 

Then $(X,\Delta+cD)$ is  $F$-pure.
\end{proposition}

\section{Non-globally-F-split klt del Pezzo surfaces}\label{s-2dim}

The goal of this section is to prove Theorem~\ref{Intro-main}. 
To this end, we first construct an unbounded sequence of klt del Pezzo surfaces, starting from a fixed log Calabi-Yau pair $(Z,C_Z)$, where $C_Z$ is a prime divisor and such that $Z$  admits exactly four singular points along $C_Z$. 
The construction was inspired by the \lq\lq hunt" in \cite{km99}.
   Note, in particular, that
it holds over any algebraically closed field. 

\begin{proposition}\label{p_4half}
Let $k$ be an algebraically closed field and let 
\[
Q_1, \ldots, Q_4 \in \mathbb P^1_k
\]
be four distinct closed points.

 Then, there exists a two-dimensional projective plt pair $(Z, C_Z)$, where $C_Z$ is a prime divisor such that the following properties hold: 
 \begin{enumerate}
\item $2(K_Z + C_Z) \sim 0$,
\item there exists an isomorphism $\iota \colon C_Z \xrightarrow{\simeq} \mathbb P_k^1$ such that 
\[
\iota_*\Diff_{C_Z} = \frac{1}{2}\sum_{i=1}^4Q_i,
\] 
and
\item for every $\epsilon \in (0,1)$, there exist birational morphisms of projective klt surfaces 
\begin{center}
\begin{tikzcd}
	Y \arrow{r}{f} \arrow{d}{g} & X \\
	Z&
\end{tikzcd}
\end{center}
such that $-K_{X}$ is ample, the strict transform $C_Y$ of $C_Z$ on $Y$ coincides with the exceptional locus of $f$, and 
there exists $b\in (1-\epsilon,1)$ such that  $K_{Y}+bC_Y=f^*K_{X}$.
%
\end{enumerate}
\end{proposition}


\begin{proof}
We divide the proof into three steps:

\setcounter{step}{0}
\begin{step}
We first construct $(Z,C_Z)$ and prove $(1)$ and $(2)$. 
\end{step}

\begin{figure}[b]

\begin{tikzpicture}
\draw[dashed] (-3,-3) -- (3,3);
\draw (-1.5,-3) -- (-1.5,3);
\draw (-0.5,-3) -- (-0.5,3);
\draw (-3,0.5) -- (3,0.5);
\draw (-3,1.5) -- (3,1.5);

\draw (-1.5,-1.5) circle(2pt)[fill=black] node[anchor = north west]{$\mathit{Q_1}$};
\draw (-0.5,-0.5) circle(2pt)[fill=black] node[anchor = north west]{ $\mathit{Q_2}$};
\draw (0.5,0.5) circle(2pt)[fill=black] node[anchor = north west]{$\mathit{Q_3}$};
\draw (1.5,1.5) circle(2pt)[fill=black] node[anchor = north west]{$\mathit{Q_4}$};

\draw (-1.5,3) node[anchor = north west]{$F_1$};
\draw (-0.5,3) node[anchor = north west]{$F_2$};
\draw (-3,0.5) node[anchor = north west]{$F_3$};
\draw (-3,1.5) node[anchor = north west]{$F_4$};

\draw (2.5,3) node[anchor = north east]{$C$};

\draw[shift={(-1.5,0.5)}, rotate=45] (-0.05,-0.05) rectangle ++(0.1,0.1)[fill=gray];
\draw[shift={(-0.5,0.5)}, rotate=45](-0.05,-0.05) rectangle ++(0.1,0.1)[fill=gray];
\draw[shift={(-1.5,1.5)}, rotate=45] (-0.05,-0.05) rectangle ++(0.1,0.1)[fill=gray];
\draw[shift={(-0.5,1.5)}, rotate=45] (-0.05,-0.05) rectangle ++(0.1,0.1)[fill=gray];
\end{tikzpicture}
\caption{ Curves on $\mathbb P^1_k\times \mathbb P^1_k$
}
\label{f_1}
\end{figure}
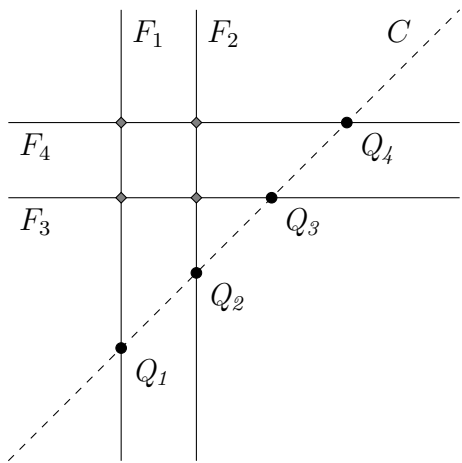

Let $C \simeq \mathbb P^1_k$ be the diagonal of $\mathbb P_k^1 \times \mathbb P_k^1$ and let
\begin{align*}
F_1&:=\pi_1^*(Q_1),\ \ F_2:=\pi_1^*(Q_2), \\
F_3&:=\pi_2^*(Q_3),\ \  F_4:=\pi_2^*(Q_4),
\end{align*}
where $\pi_i \colon \mathbb P_k^1 \times \mathbb P_k^1 \to \mathbb P_k^1$ denotes the $i$-th projection, for $i=1,2$ (see Figure \ref{f_1}). 

%
%
%
%

In particular, the plt pair 
\[
\Big(\mathbb P_k^1 \times \mathbb P_k^1, C+\frac{1}{2}\sum_{i=1}^4F_i\Big)
\]
satisfies $2\big(K_{\mathbb P_k^1 \times \mathbb P_k^1}+C+\frac{1}{2}\sum_{i=1}^4F_i\big) \sim 0.$ 

Let
\[
\varphi \colon S \longrightarrow \mathbb P_k^1 \times \mathbb P_k^1
\]
be the blow-up along the four points in the set  
\[
(F_1 \cup F_2) \cap (F_3 \cup F_4).
\]

By abuse of notation, we denote the proper transforms of $C,F_1,\dots,F_4$ on $S$ by the same symbols. 
We have 
\[
2\Big(K_{S}+C+\frac{1}{2}\sum_{i=1}^4F_i\Big) \sim 0.
\]

\definecolor{lightgray}{rgb}{0.75, 0.75, 0.75}

Note that $F_1, \ldots, F_4$ are pairwise disjoint $(-2)$-curves on $S$.  Let 
\[
\psi \colon S \rightarrow Z
\]
be the contraction of these curves and let 
 $C_Z:=\psi_*C$. 
We obtain 
\[
2\left(K_Z+C_Z\right)=\psi_*\Big(2\Big(K_{S}+C+\frac{1}{2}\sum_{i=1}^4F_i\Big)\Big) \sim 0.
\]
Further, $(Z,C_Z)$ is plt and $\mathrm{Diff}_{C_Z} = \frac{1}{2}\sum_{i=1}^4Q_i$. Thus, $(1)$ and $(2)$ follow. 

Note that, by the Nakai--Moishezon criterion, $C_Z$ is ample. 
Moreover, $Z$ has exactly four singular points, and all of them lie on the curve $C_Z$.

\begin{step}
We now construct  $X$ and $Y$.
\end{step}

Fix an integer $n \in \mathbb Z$ such that 
\[
n \geq 4\qquad \text{and} \qquad \frac{2}{n-\frac{7}{2}}<\epsilon.
\] 
Let $\overline{S} \to S$ be the birational morphism obtained by blowing-up $n$-times the point $Q_1 \in C$: by the point $Q_1$ in the blow-up of $S$, we mean the unique intersection of the strict transform of $C$ with the exceptional curve.
Again, by abuse of notation, we denote the strict transforms of $C, F_1, \ldots, F_4$ by the same symbols. Further, we denote the exceptional curves on $\overline{S}$ by $E_1, \ldots, E_n$, where $E_i$ is the exceptional curve of the $i$-th blow-up (see Figure \ref{f_2}).

It follows that 
\[
2\Big(K_{\overline{S}}+C+\frac{1}{2}\sum_{i=1}^4F_i+\frac{1}{2}\sum_{i=1}^nE_i\Big) \sim 0.
\]

\begin{figure}[t]
\begin{tikzpicture}
\draw (-6,2) -- (4,2);

\draw[shift={(-4,2)}, rotate=-20, line width = 1.5pt] (0,0.7) -- (0,-1.5);
\draw[shift={(-2,2)}, rotate=-20] (0,0.7) -- (0,-1.5);
\draw[shift={(0,2)}, rotate=-20] (0,0.7) -- (0,-1.5);
\draw[shift={(2,2)}, rotate=-20] (0,0.7) -- (0,-1.5);

\draw[shift={(-4.1,0.5)}, rotate=20, line width = 1.5pt] (0,1) -- (0,-1);
\draw[shift={(-4.1,-0.3)}, rotate=-20, line width = 1.5pt] (0,1) -- (0,-1);

\draw[shift={(-4.1,-1.55)}] node{$\vdots$}; 

\draw[shift={(-4.1,-3)}, rotate=20, line width = 1.5pt] (0,1) -- (0,-1);
\draw[shift={(-4.1,-3.8)}, rotate=-20] (0,1) -- (0,-1);

\draw (-4,2) circle(2pt)[fill=black] node[anchor = south east]{$\mathit{Q_1}$};
\draw (-2,2) circle(2pt)[fill=black] node[anchor = south east]{ $\mathit{Q_2}$};
\draw (0,2) circle(2pt)[fill=black] node[anchor = south east]{$\mathit{Q_3}$};
\draw (2,2) circle(2pt)[fill=black] node[anchor = south east]{$\mathit{Q_4}$};

\draw (-5.4,1.3) node{$E_n$};
\draw (-2.8,1.3) node{$F_2$};
\draw (-0.8,1.3) node{$F_3$};
\draw (1.2,1.3) node{$F_4$};

\draw (-5.4,0.3) node{$E_{n-1}$};
\draw (-5.4,-2.9) node{$E_{1}$};

\draw (-5.4,-0.7) node{$E_{n-2}$};
\draw (-5.4,-4) node{$F_1$};

\draw (4,2) node[anchor = north east]{$C$};
\end{tikzpicture}
\caption{ Curves on $\overline S$
} 
\label{f_2}
\end{figure}
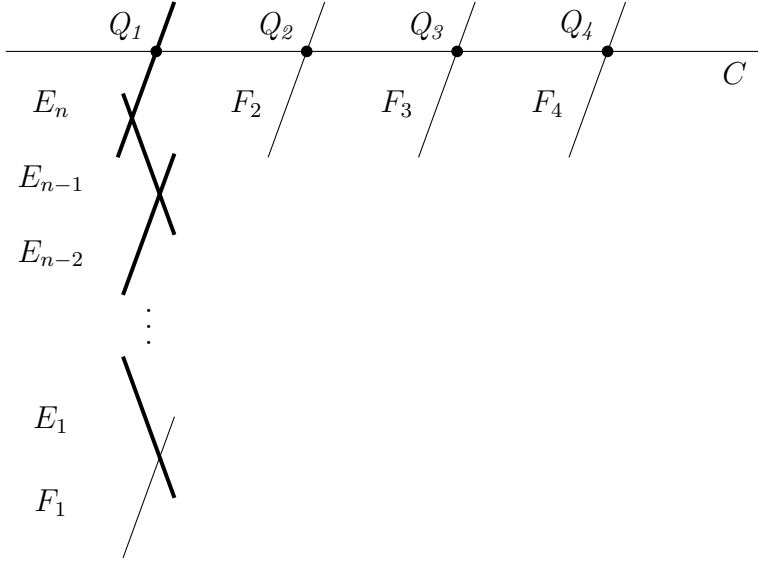

Let 
\[
h \colon \overline{S} \longrightarrow Y
\]
be the birational contraction of $F_2$, $F_3$, $F_4$ and of the chain of curves $E_{n-1}, \ldots, E_1, F_1$. 
We can construct such a birational morphism
by running the MMP over $Z$  \cite[Theorem~6.5]{tanaka12}
with respect to the pair
\[
\Big(\overline{S}, \ C+\frac{1}{2}E_n+\frac{3}{4}\Big(\sum_{i=1}^4F_i+\sum_{i=1}^{n-1}E_i\Big)\Big).
\]

Let $C_{Y} := h_*C$ and $E_{Y} := h_* E_n$. Then, there exists a  birational morphism $g \colon Y \to Z$, whose exceptional divisor coincides with  $E_Y$.
 In particular, $C_Y$ is the strict transform of $C_Z$ on $Y$. Thus, $C_Y$ is  a smooth rational curve and it can be easily checked that 
\[
C_Y^2=\frac{7}{2}-n.
\]
Because of our choice of $n$, we have $C_Y^2<0$.  We denote by  $f \colon Y \to X$  the contraction of $C_Y$. 
Summarising, we have the following diagram:
%
%
%
%
%

\begin{center}
\begin{tikzcd}
	\overline{S} \arrow{r}{h} \arrow{d} & Y \arrow{r}{f} \arrow{d}{g} & X \\
	S \arrow{r} \arrow{d} & Z & \\
	\mathbb P_k^1 \times \mathbb P_k^1. & &
\end{tikzcd}
\end{center} 

By taking a pullback to $\overline{S}$, it is easy to check that
\[
K_{Y}+C_{Y}+\frac{1}{2}E_{Y} = g^*(K_Z+C_Z) \sim_{\mathbb{Q}} 0.
\]
Since $(Z,C_Z)$ is plt,  we have that also $(Y,C_Y+\frac{1}{2}E_{Y})$ is plt, and $Y$ is klt. 

\begin{step}
We now prove $(3)$. 
\end{step}

 Let $a \in \Q_{>0}$ be such that
\[
g^*C_Z = C_Y + aE_Y.
\]
Note that
\[
g^*C_Z-\delta E_Y=C_Y + (a-\delta)E_Y 
\]
is ample for any $0 < \delta \ll 1$, and so $f_*E_Y$ is ample. Since
\[
-(K_Y + C_Y) \sim_{\mathbb Q} \frac{1}{2}E_Y,
\]
we have that
\[
-K_X = -f_*(K_Y+C_Y) \sim_{\mathbb Q}  \frac{1}{2}f_*E_Y
\]
is ample.
Let $b\in \mathbb Q$ be such that 
\[
K_Y + bC_Y = f^*K_X.
\]
By intersecting with $C_Y$, we have 
\[
b = \frac{K_Y \cdot C_Y}{-C_Y^2} = 1 - \frac{-(K_Y+C_Y)\cdot C_Y}{-C_Y^2}. 
\]

Since $K_Y + C_Y + \frac{1}{2}E_Y \equiv 0$ and $C_Y \cdot E_Y > 0$, we have
\[
0 < -(K_Y + C_Y) \cdot C_Y \leq 2.
\]
Thus, because of our choice of $n$,  we obtain 
\[
0<1-b=\frac{-(K_Y+C_Y)\cdot C_Y}{-C_Y^2} \leq \frac{2}{n-\frac{7}{2}}<\epsilon.
\]
In particular, $X$ is klt and $(3)$ follows. 
\end{proof}

We can now prove our main theorem: 

\begin{proof}[Proof of Theorem~\ref{Intro-main}]
Fix an algebraically closed field $k$ of characteristic $p>0$. 
If $p=2$, then the smooth cubic surface of Fermat type satisfies the required properties 
\cite[Example~5.5]{hara98a}. 
Thus, we may assume $p \geq 3$.

Let $Q_1,\ldots, Q_4 \in \mathbb P^1_k$ be as in Lemma \ref{l_P1} and let $X,Y,Z$ and $b$ be as in Proposition \ref{p_4half} for some $0 < \epsilon \ll 1$. In particular, $X$ is klt. 
By $(1)$ and  $(2)$ of Proposition \ref{p_4half}  and Lemma \ref{l_adjunction}, it follows that the pair $(Z,C_Z)$ is not globally $F$-split. By Theorem \ref{t-fpure-threshold}, there exists a rational number $\beta$, independent of $\epsilon$, 
such that $0<\beta<1$ and the pair $(Z, \beta C_Z)$ is not globally $F$-split as well.

Thus, we may assume $\epsilon < 1 - \beta$, and in particular $b > \beta$. Since $(Z, \beta C_Z)$ is not globally $F$-split, Lemma~\ref{l-push-GFS} implies that $(Y, \beta C_{Y})$ is not globally $F$-split. Thus,  $(Y, bC_{Y})$ is not globally $F$-split and, by   Lemma  \ref{l-pull-GFS}, also $X$ is not globally $F$-split. 
\end{proof}

\section{Non-F-pure canonical threefolds}\label{s-3dim}

The goal of this section is to prove Theorem~\ref{Intro-main2}. 
We begin with the following Lemma. 

\begin{lemma}\label{l-canonical-check}
Let $k$ be an algebraically closed field of characteristic $p>5$. 
Let $f(x, y, z)\in k[x, y, z]$ be a homogeneous cubic polynomial such that 
${\Proj}\,k[x, y, z]/(f(x, y, z))$ is a smooth elliptic curve. 

Then, for every $n\in \Z_{\geq 0}$, 
\[
X_n=\Spec\,k[x, y, z, w]/(f(x, y, z)+w^n)
\]
has at most canonical singularities. 
\end{lemma}

\begin{proof} We divide the proof into three steps:
\setcounter{step}{0}

\begin{step}\label{step-origin}
We first show that $X_n$ is smooth outside the origin and, in particular,  $X_n$ is smooth if $n\le 1$. 
\end{step}

Let $(a, b, c, d) \in X_n \subset \mathbb A^4_k$ be a singular point.
Since $f$ is homogeneous of degree $3$, we have 
\[
x\frac{\partial f}{\partial x}+y\frac{\partial f}{\partial y}+z\frac{\partial f}{\partial z}=3f.
\]
Since $p>3$, a singular point $(a, b, c, d)$ satisfies the following equation: 
\[
f(a, b, c)=\frac{\partial f}{\partial x}(a, b, c)=\frac{\partial f}{\partial y}(a, b, c)=\frac{\partial f}{\partial z}(a, b, c)=0.
\]
This implies $a=b=c=0$, hence $d=0$, as claimed. 

\begin{step}\label{step-X2} We now show that  $X_2$ is terminal. 
\end{step}

Since $p>3$, by writing the cubic $f$ in  Weierstrass form, we may assume that
there exist $a,b\in k$ such that
\[
f(x, y, z)+w^2=y^2z+x^3+axz^2+bz^3+w^2.
\]
After taking a linear transformation $x\mapsto x+cz$ for some $c \in k$, 
we may write 
\[
f(x, y, z)+w^2=y^2z+x^3+\alpha x^2z+\beta xz^2+\gamma z^3+w^2
\]
for some $\alpha,\beta,\gamma\in k$ with $\gamma \neq 0$. 
Consider the hypersurface $S$ in $X_2$ obtained by cutting with the hyperplane 
$\{x=0\}$, i.e. 
\[
S=\Spec\,k[y, z, w]/(y^2z+\gamma z^3+w^2).
\]
Then $S$ admits a unique singularity which is  du Val  of type $D_4$, and in particular $S$ is canonical. 
Since $p>5$, $S$ is strongly $F$-regular \cite[Theorem 1.1]{hara98} and inversion of adjunction \cite[Theorem A]{Das15} implies  that the pair $(X_2, S)$ is plt. 
Since $S$ is a Cartier divisor, it follows that $X_2$ is terminal, as claimed.

\begin{step}\label{step-blowup}
We now show that $X_n$ is canonical for all $n\in \Z_{\ge 0}$. 
\end{step}

We may assume $n \geq 3$. 
Consider the blow-up $f\colon Y \to X_n$ at the origin. 
Then $Y$ is covered by four affine open subsets: 
\begin{eqnarray*}
Y&=&\Spec\,k[x, y, z, w]/(f(1, y, z)+x^{n-3}w^{n})\\
&\cup&\Spec\,k[x, y, z, w]/(f(x, 1, z)+y^{n-3}w^{n})\\ 
&\cup& \Spec\,k[x, y, z, w]/(f(x, y, 1)+z^{n-3}w^{n})\\
&\cup& \Spec\,k[x, y, z, w]/(f(x, y, z)+w^{n-3}).
\end{eqnarray*}
We first  show that the first three affine open subsets are smooth. 
By  symmetry, it is enough to consider the first one. By Step 1,  
we just need to consider points in the exceptional locus $\{x=0\}$. 
Since $n \geq 3$,  a singular point $(x,y,z,w)$ satisfies the  equations: 
\[
f(1, y, z)=\frac{\partial f(1, y, z)}{\partial y}=\frac{\partial f(1, y, z)}{\partial z}=0.
\]
Since $\Spec\,k[y, z]/(f(1, y, z))$ is a smooth affine curve, there are no such points. 
Thus, the singular locus of $Y$ is contained in the last affine open subset 
\[
\Spec\,k[x, y, z, w]/(f(x, y, z)+w^{n-3}) \simeq X_{n-3}.
\]

Moreover, the multiplicity of $X_n$ at the origin is three. Thus, $f$ is  crepant, i.e.   
\[
K_Y=f^*K_{X_n}.
\]
Thus, it is enough to show that $X_{n-3}$ is canonical. Repeating the same procedure finitely many times, 
Step 1 and Step 2 imply the claim. 
This completes the proof of the Lemma. 
\end{proof}

\begin{lemma}\label{l-Fpure-check}
Let $k$ be an algebraically closed field of characteristic ${p>0}$. 
Let $f(x, y, z)\in k[x, y, z]$ be a homogeneous cubic polynomial such that 
${\Proj}\,k[x, y, z]/(f(x, y, z))$ is a smooth elliptic curve.

Let $n \in \Z$ be such that $n \geq p$, and let 
\[
X_n=\Spec\,k[x, y, z, w]/(f(x, y, z)+w^n).
\]
Then, $X_n$ is $F$-pure if and only if 
${\Proj}\,k[x, y, z]/(f(x, y, z))$ is an ordinary elliptic curve. 
\end{lemma}

\begin{proof}
We consider the following conditions. 
\begin{enumerate}
\item{${\Proj}\,k[x, y, z]/(f(x, y, z))$ is an ordinary elliptic curve. }
\item{${\Spec}\,k[x, y, z]/(f(x, y, z))$ is $F$-pure. }
\item{$f(x, y, z)^{p-1} \not\in (x^p, y^p, z^p)k[x, y, z].$}
\end{enumerate}
\begin{enumerate}
\item[(a)]{$X_n$ is $F$-pure.}
\item[(b)]{$(f(x, y, z)+w^n)^{p-1} \not\in (x^p, y^p, z^p, w^p)k[x, y, z, w].$}
\end{enumerate}
It is well known that (1), (2) and (3) (respectively, (a) and (b)) are equivalent. Since, by assumption: 
$n\ge p$, it is easy to check that (3) and (b) are equivalent. Thus, the claim follows.
\end{proof}


We can now prove Theorem~\ref{Intro-main2}.

\begin{proof}[Proof of Theorem~\ref{Intro-main2}]
If $p \leq 5$, then 
\[
X:=\Spec k[x, y, z, w]/(x^2+y^3+z^5)
\] is canonical but not $F$-pure \cite[(4.4)]{hara98}.  

Thus, we may assume $p>5$. 
Let $f(x, y, z)\in k[x, y, z]$ be a homogeneous cubic polynomial such that 
${\Proj}\,k[x, y, z]/(f(x, y, z))$ is a supersingular elliptic curve. 
Let 
\[
X:=k[x, y, z, w]/(f(x, y, z)+w^p).
\]
By Lemma~\ref{l-canonical-check}, $X$ is canonical. 
By Lemma~\ref{l-Fpure-check}, $X$ is not $F$-pure. 
\end{proof}


\bibliographystyle{amsalpha}
\bibliography{Library}

\newcommand{\etalchar}[1]{$^{#1}$}
\providecommand{\bysame}{\leavevmode\hbox to3em{\hrulefill}\thinspace}
\providecommand{\MR}{\relax\ifhmode\unskip\space\fi MR }
\providecommand{\MRhref}[2]{%
  \href{http://www.ams.org/mathscinet-getitem?mr=#1}{#2}
}
\providecommand{\href}[2]{#2}
\begin{thebibliography}{HNBWZ15}

\bibitem[BMS08]{BMS08}
M.~Blickle, M.~Musta{\c{t}}{\v{a}}, and K.E. Smith, \emph{Discreteness and
  rationality of {$F$}-thresholds}, Michigan Math. J. \textbf{57} (2008),
  43--61, Special volume in honor of Melvin Hochster.

\bibitem[BSTZ10]{BSTZ10}
M.~Blickle, K.~Schwede, S.~Takagi, and W.~Zhang, \emph{Discreteness and
  rationality of {$F$}-jumping numbers on singular varieties}, Math. Ann.
  \textbf{347} (2010), no.~4, 917--949.

\bibitem[CGS14]{cgs14}
P.~Cascini, Y.~Gongyo, and K.~Schwede, \emph{Uniform bounds for strangly
  ${F}$-regular surfaces}, Trans. Amer. Math. Soc. (to appear) (2014).

\bibitem[CTW15]{CTW15b}
P.~Cascini, H.~Tanaka, and J.~Witaszek, \emph{On log del {P}ezzo surfaces in
  large characteristic}, Preprint (2015).

\bibitem[Das15]{Das15}
O.~Das, \emph{On strongly {$F$}-regular inversion of adjunction}, J. Algebra
  \textbf{434} (2015), 207--226.

\bibitem[Har98a]{hara98a}
N.~Hara, \emph{A characterization of rational singularities in terms of
  injectivity of {F}robenius maps}, Amer. J. Math. \textbf{120} (1998), no.~5,
  981--996.

\bibitem[Har98b]{hara98}
N.~Hara, \emph{Classification of two-dimensional {$F$}-regular and {$F$}-pure
  singularities}, Adv. Math. \textbf{133} (1998), no.~1, 33--53.

\bibitem[Har06]{hara06}
\bysame, \emph{F-pure thresholds and {F}-jumping exponents in dimension two},
  Math. Res. Lett. \textbf{13} (2006), no.~5-6, 747--760, With an appendix by
  Paul Monsky.

\bibitem[Her12]{hernandez12}
D.J. Hern\'andez, \emph{{$F$}-purity of hypersurfaces}, Math. Res. Lett.
  \textbf{19} (2012), no.~2, 389--401.

\bibitem[HH90]{hh90}
M.~Hochster and C.~Huneke, \emph{Tight closure, invariant theory, and the
  {B}rian\c con-{S}koda theorem}, J. Amer. Math. Soc. \textbf{3} (1990), no.~1,
  31--116.

\bibitem[HNBWZ15]{HNWZ15}
D.J. Hern\'andez, L.~N{\'u}{\~n}ez-Betancourt, E.E. Witt, and W.~Zhang,
  \emph{{$F$}-pure thresholds of homogeneous polynomials}, Michigan Math. J.
  (to appear) (2015).

\bibitem[HX15]{hx13}
C.~Hacon and C.~Xu, \emph{On the three dimensional minimal model program in
  positive characteristic}, J. Amer. Math. Soc. \textbf{28} (2015), no.~3,
  711--744.

\bibitem[K{\etalchar{+}}92]{Kollaretal}
J.~Koll{\'a}r et~al., \emph{Flips and abundance for algebraic threefolds},
  Soci{\'e}t{\'e} Math{\'e}matique de France, Paris, 1992.

\bibitem[KM98]{KM98}
J.~Koll{\'a}r and S.~Mori, \emph{Birational {G}eometry of {A}lgebraic
  {V}arieties}, Cambridge {T}racts in {M}athematics, vol. 134, Cambridge
  University Press, 1998.

\bibitem[KM99]{km99}
S.~Keel and J.~M\textsuperscript{c}Kernan, \emph{Rational curves on
  quasi-projective surfaces}, Mem. Amer. Math. Soc. \textbf{140} (1999),
  no.~669, viii+153.

\bibitem[PST14]{PST14}
Z.~Patakfalvi, K.~Schwede, and K.~Tucker, \emph{Notes for the workshop on
  positive characteristic algebraic geometry}, arXiv:1412.2203 (2014).

\bibitem[Sch08]{schwede08}
K.~Schwede, \emph{Generalized test ideals, sharp {$F$}-purity, and sharp test
  elements}, Math. Res. Lett. \textbf{15} (2008), no.~6, 1251--1261.

\bibitem[Smi00]{smith00}
K.~E. Smith, \emph{Globally {F}-regular varieties: applications to vanishing
  theorems for quotients of {F}ano varieties}, Michigan Math. J. \textbf{48}
  (2000), 553--572, Dedicated to William Fulton on the occasion of his 60th
  birthday.

\bibitem[SS10]{schwedesmith10}
K.~Schwede and K.~E. Smith, \emph{Globally {$F$}-regular and log {F}ano
  varieties}, Adv. Math. \textbf{224} (2010), no.~3, 863--894.

\bibitem[Tan14]{tanaka12}
H.~Tanaka, \emph{Minimal models and abundance for positive characteristic log
  surfaces}, Nagoya Math. J. \textbf{216} (2014), 1--70.

\bibitem[TW04]{tw04}
S.~Takagi and K.~Watanabe, \emph{On {F}-pure thresholds}, J. Algebra
  \textbf{282} (2004), no.~1, 278--297.

\bibitem[Wat91]{watanabe91}
K.~Watanabe, \emph{{$F$}-regular and {$F$}-pure normal graded rings}, J. Pure
  Appl. Algebra \textbf{71} (1991), no.~2-3, 341--350.

\end{thebibliography}

\end{document}